\documentclass{birkjour}
 \newtheorem{thm}{Theorem}[section]
 
 \newtheorem{lem}[thm]{Lemma}
 
 \theoremstyle{definition}
 
 \theoremstyle{remark}

 \numberwithin{equation}{section}

\newcommand{\field}[1]{\mathbb{#1}}
\newcommand{\MM}{\field{M}}
\newcommand{\SL}{\mathrm{SL}}
\newcommand{\pSL}{\mathrm{(P)SL}}
\newcommand{\Z}{\mathrm{Z}}

\begin{document}
\title[Special rank one groups]
{Special rank one groups with abelian\\ 
unipotent subgroups are perfect}

\author[Anja Steinbach]{Anja Steinbach}

\address{%
Justus-Liebig-Universit\"at Gie{\ss}en\\
Mathematisches Institut\\
Arndtstra{\ss}e 2\\
D 35392 Gie{\ss}en\\
Germany}

\email{Anja.Steinbach@math.uni-giessen.de}



\date{August 31, 2011}

\begin{abstract}
We prove that special (abstract) rank one groups with abelian unipotent subgroups of size at least 4 are perfect.
\end{abstract}

\maketitle
\section{Introduction}

J. Tits \cite{Tits-Durham} defined Moufang sets in order to axiomatize the linear algebraic groups of relative rank one.
A closely related concept of so-called (abstract) rank one groups has been introduced by F. G. Timmesfeld \cite{TT-Buch}.

Here a group $X$ is an (abstract) rank one group with abelian unipotent subgroups $A$ and $B$, 
if $X = \langle A,B\rangle$ with $A$ and $B$ different abelian subgroups of $X$, and (writing $A^b = b^{-1}Ab$)
\begin{align*}
&\text{for each $1 \neq a \in A$, there is an element $1 \neq b \in B$ such that $A^b = B^a$,} \\
&\text{and vice versa.}
\end{align*}
In an (abstract) rank one group $X$ with abelian unipotent subgroups $A$ and $B$,
the element $1 \neq b \in B$ with $A^b = B^a$ is uniquely determined for each $1 \neq a \in A$ (as $A \neq B$) and denoted by $b(a)$. We say $X$ is special, if $b(a^{-1}) = b(a)^{-1}$ for all $1 \neq a  \in A$ 
(see Timmesfeld \cite[I (2.2), p. 17]{TT-Buch}).
In this note we prove 
\begin{thm}\label{Main-Result}
Let $X$ be a special (abstract) rank one group with abelian uni\-potent subgroups $A$ and $B$ 
and set $H = N_X(A) \cap N_X(B)$. Then $A = [A,H]$ provided that $|A| \geq 4$.
\end{thm}

By \cite[I (2.10), p. 25]{TT-Buch} this yields that a special (abstract) rank one group with abelian unipotent subgroups is either quasi-simple or isomorphic to $\SL_2(2)$ or $\pSL_2(3)$, as was conjectured by Timmesfeld \cite[Remark, p.~26]{TT-Buch}. 

We remark that $X / \Z(X)$ is the projective group of a Moufang set and that is the point of view of T. De Medts, 
Y. Segev and K. Tent, who proved Timmesfeld's conjecture first. They deduced it from suitable identities in special Moufang sets $\MM(U, \tau)$, see \cite[Theorem 1.12]{DM-S-T}, without the assumption that $U$ is abelian.

In Timmesfeld's (quasi) simplicity criterion for groups generated by abstract root subgroups \cite[II (2.14)]{TT-Buch}
only (abstract) rank one groups with abelian unipotent subgroups are involved. My aim was to simplify his criterion via a short, elementary and self-contained proof of Theorem \ref{Main-Result}. The argument given below does not need a case differentiation whether $A$ is an elementary abelian 2-group or not. We show that $a_1[A,H] = a_2[A,H]$ for all $1 \neq a_1, a_2 \in A$ with $a_1a_2 \neq 1$.

\section{The proof of Theorem \ref{Main-Result}}
Let $X$ be a special (abstract) rank one group with abelian unipotent subgroups $A$ and $B$.
For each $1 \neq a \in A$, we set $n(a) := ab(a)^{-1}a$. 
Then $B^{n(a)} = A$. As $b(a)^{-1} = b(a^{-1})$, also $A^{n(a)} = B$. Thus $n(a)n(a') \in H := N_X(A) \cap N_X(B)$, for
all $1 \neq a, a' \in A$.
For $1 \neq a \in A$, we have
\begin{equation}\label{force-h}
B^{a^{-1}n(a)} = B^a
\end{equation}

\begin{lem} \label{useful-equation}
We have $B^{a_1a_2n(a_2^{-1})a_2} = 
B^{a_2a_1n(a_1^{-1})a_1}$, for all $1 \neq a_1, a_2 \in A$.
\end{lem}

\begin{proof}
By the definition of $n(a_2^{-1})$ we have $B^{a_1a_2n(a_2^{-1})a_2} = B^{a_1b(a_2^{-1})^{-1}}$.
As $X$ is special, the left hand side of the claim equals $A^{b(a_1)b(a_2)}$.
Similarly, the right hand side is $A^{b(a_2)b(a_1)}$. 
As $B$ is abelian, the claim follows.
\end{proof}

\begin{lem} \label{equation-in-A}
We have $a_1 \in a_2[A,H]$, for all $1 \neq a_1, a_2 \in A$ with $a_1a_2 \neq 1$.
\end{lem}

\begin{proof}
Let $1 \neq a_1, a_2 \in A$ with $a_1a_2 \neq 1$. 
We set $h := n(a_1a_2)n(a_2^{-1}) \in H$.
By \eqref{force-h} we have $B^{a_1a_2n(a_2^{-1})a_2} = B^{a_2^{-1} a_1^{-1} h a_2}$.
As $A$ is abelian and $B^h = B$, the last term equals $B^{a_1^{-1}[a_1^{-1},h][h,a_2]}$.

Similarly, $B^{a_2a_1n(a_1^{-1})a_1} = B^{a_2^{-1}a_0}$ for some $a_0 \in [A,H]$.
We apply Lemma \ref{useful-equation}. As $N_A(B) = 1$, we obtain $a_1^{-1}[A,H] = a_2^{-1}[A,H]$. Thus the claim holds.
\end{proof}

When $[A,H] = 1$, then Lemma \ref{equation-in-A} implies that 
$A \subseteq \{1, a, a^{-1}\}$, where $1 \neq a \in A$; i.e., $|A| \leq 3$.
Thus for $|A| \geq 4$, we may choose $1 \neq a \in [A,H]$. By Lemma \ref{equation-in-A}, we obtain
$A \subseteq a[A,H] \cup \{1, a^{-1}\} \subseteq [A,H]$, as desired.

\end{document}